 \documentclass[final,1p,times,twocolumn,authoryear]{elsarticle}
\usepackage{amsmath,amsthm,amssymb}
\usepackage{bbm}
\usepackage{natbib}
\usepackage{xcolor}
\usepackage[colorlinks=true, citecolor=blue, linkcolor=red, urlcolor=blue]{hyperref}

 \usepackage{lineno}

\journal{Arxiv}

\theoremstyle{plain}
\newtheorem{theorem}{Theorem}

\theoremstyle{definition}
\newtheorem{example}{Example}

\begin{document}

%\linenumbers

\begin{frontmatter}



\title{\LARGE $L_2$-norm posterior contraction in Gaussian models \\ with unknown variance} %% Article title


\author{Seonghyun Jeong} %% Author name
\ead{sjeong@yonsei.ac.kr}

\affiliation{organization={Department of Statistics and Data Science, Department of Applied Statistics, Yonsei University},%Department and Organization
	addressline={50 Yonsei-ro, Seodaemun-gu}, 
	city={Seoul},
	postcode={03722}, 
	country={Republic of Korea}}

%% Abstract
\begin{abstract}

The testing-based approach is a fundamental tool for establishing posterior contraction rates. Although the Hellinger metric is attractive owing to the existence of a desirable test function, it is not directly applicable in Gaussian models, because translating the Hellinger metric into more intuitive metrics typically requires strong boundedness conditions. When the variance is known, this issue can be addressed by directly constructing a test function relative to the $L_2$-metric using the likelihood ratio test. However, when the variance is unknown, existing results are limited and rely on restrictive assumptions.
To overcome this limitation, we derive a test function tailored to an unknown variance setting with respect to the $L_2$-metric and provide sufficient conditions for posterior contraction based on the testing-based approach. We apply this result to analyze high-dimensional regression and nonparametric regression.

\end{abstract}



%% Keywords

\end{frontmatter}



\section{Introduction}

We suppose that the observation vector $y\in\mathbb R^n$ follows the Gaussian model given by
\begin{align}
y\sim \text{N}_n(\mu_0,\sigma_0^2 I_n),
\label{eqn:model}
\end{align}
where $\mu_0=(\mu_{01},\dots,\mu_{0n})^T\in\mathbb R^n$ and $\sigma_0^2>0$ are the true parameters for data generation. Although $\sigma_0^2$ is typically assumed to be fixed and independent of $n$, we allow it to vary with $n$ for further flexibility.
Several interesting models can be incorporated into this framework by appropriately specifying $\mu_0$. For example, if $\mu_0$ is a sparse vector, one obtains a sparse normal mean model; if $\mu_{0i} = x_i^T \beta_0$ with covariates $x_i \in \mathbb{R}^p$ and a (possibly sparse) coefficient vector $\beta_0 \in \mathbb{R}^p$, the model corresponds to (sparse) regression; and if $\mu_{0i} = f_0(x_i)$ for a function $f_0 : \mathbb R^p \rightarrow \mathbb{R}$, it leads to a nonparametric regression model. Posterior contraction rates in these models have been extensively studied under the assumption that
 $\sigma_0^2$
is known and fixed in the literature  \citep[e.g.,][]{castillo2012needles,castillo2015bayesian,van2008rates,rovckova2020posterior,polson2018posterior}. In contrast, the case of an unknown $\sigma_0^2$ is often addressed in a problem-specific manner \citep[e.g.,][]{yoo2016supremum,ning2020bayesian,jeong2021unified,song2023nearly}.
A key reason for this limitation is that assuming a known $\sigma_0^2$ simplifies the problem setup, thereby allowing for the establishment of optimal posterior contraction rates by directly exploiting the model structure under mild conditions with appropriately tailored priors \citep[e.g.,][]{castillo2012needles,castillo2015bayesian}. Another contributing factor is that, unlike in the known $\sigma_0^2$ case, a test function required for testing-based posterior contraction theory is available only in a somewhat restrictive form.

The testing-based approach to posterior contraction was formalized by \citet{ghosal2000convergence} and \citet{ghosal2007convergence}, showing the fundamental role of tests, as in Schwartz's theory of posterior consistency \citep{schwartz1965bayes}. Typically, though not necessarily, a global test over a sieve---a suitable subset of the parameter space---is constructed from local tests on small subsets that are sufficiently separated from the true parameters, combined with an appropriately controlled covering number, often measured by metric entropy.
Consequently, obtaining such a local test function is crucial. Notably, the Hellinger metric and its averaged version always allow for the construction of test functions with exponentially small type-I and type-II errors, regardless of the specified model formulation \citep{lecam1973convergence,birge1983robust}.
However, using the Hellinger metric is usually unsatisfactory for Gaussian models, because it is not directly related to a more intuitive metric of the parameters without imposing strong boundedness conditions \citep{ghosal2007convergence}. To address this issue, when $\sigma_0^2$ is known, a direct local test can be constructed relative to the $L_2$-norm of $\mu_0$ using the likelihood ratio test \citep{birge2006model}. However, when $\sigma_0^2$ is unknown, available results are limited. Lemma~8.27 of \citet{ghosal2017fundamentals} extends the basic result for the known-variance case, but the testing error is not uniformly controlled over small local subsets, making it difficult to apply the result across the entire parameter space without truncation. Proposition~S2 of \citet{naulet2018some} provides a refined result, but their test is not suitable for achieving optimal contraction rates when $\sigma_0^2$ varies with $n$. This study aims to complement the existing results.

Our primary goal is to establish a test function for the general form in \eqref{eqn:model}. To achieve this, we first derive a local test function that distinguishes the true parameters from alternatives lying in small balls that are sufficiently separated in the $L_2$-norm. Following the standard entropy approach, we combine these local tests with metric entropy bounds to construct a global test over an appropriately chosen sieve. By employing the testing-based strategy, we establish sufficient conditions for $L_2$-norm posterior contraction by ensuring that the prior concentrates on a Kullback-Leibler neighborhood of the true parameters while assigning an exponentially small mass outside the sieve. Unlike \citet{ghosal2017fundamentals} and \citet{naulet2018some}, our results offer a framework for achieving optimal contraction rates with priors supported on $(0,\infty)$ even when $\sigma_0^2$ varies with $n$. Section~\ref{sec:main} presents the main results, and Section~\ref{sec:example} discusses the applications to high-dimensional regression and nonparametric regression. Section~\ref{sec:disc} concludes the study with a discussion.

\section{Main results}
\label{sec:main}
For $\mu\in\mathbb R^n$ and $\sigma^2>0$, let $\mathbb E_{\mu,\sigma}$ denote the expectation operator under $\textnormal{N}_n(\mu,\sigma^2 I_n)$ and $P_{\mu,\sigma}$ be the corresponding probability measure. We denote by $\lVert\cdot\rVert_p$ the $L_p$-norm of a vector. For any $\mu_1,\mu_2\in\mathbb R^n$ and $\sigma_1,\sigma_2\in(0,\infty)$, we define the metric $d$ as
$d^2((\mu_1,\sigma_1),(\mu_2,\sigma_2))=n^{-1}\lVert \mu_1-\mu_2 \rVert_2^2 + |\sigma_1-\sigma_2|^2$, which can be viewed as the $L_2$-distance between the vectors $(n^{-1/2}\mu_1^T,\sigma_1)^T$ and $(n^{-1/2}\mu_2^T,\sigma_2)^T$.
A key contribution of this study is the construction of a test function that distinguishes the true parameters from alternative values that lie in a small ball sufficiently separated from the true parameters with respect to $d$.

\begin{theorem}[Local test]
	\label{lmm:test}
For any $\epsilon\in(0,\sigma_0)$ and any $(\mu_1,\sigma_1)$ satisfying $d((\mu_1,\sigma_1),(\mu_0,\sigma_0))\ge \epsilon$, there exists a test $\varphi_n\in\{0,1\}$ such that, for a universal constant $K>0$,
	\begin{align*}
		\mathbb E_{\mu_0,\sigma_0} \varphi_n \le e^{-Kn\epsilon^2/\sigma_0^2},\quad 	\sup_{(\mu,\sigma):d((\mu,\sigma),(\mu_1,\sigma_1))\le\epsilon/6}\mathbb E_{\mu,\sigma} (1-\varphi_n) \le e^{-Kn\epsilon^2/\sigma_0^2}.
	\end{align*}
\end{theorem}

\begin{proof}
	Let $y\in\mathbb R^n$ be a Gaussian random vector.
	For $M_1>1$, we separate the range of the parameters as follows.
%	\begin{itemize}
%		\item \it{Case 1}: $\sigma_1\ge M_1\sigma_0$.
%		\item \it{Case 2}: $\sigma_1 < M_1\sigma_0$ and $7\lVert\mu_1-\mu_0\rVert_2^2 > n |\sigma_1-\sigma_0|^2$.
%		\item \it{Case 3}: $\sigma_0\le\sigma_1 < M_1 \sigma_0$ and $7\lVert\mu_1-\mu_0\rVert_2^2 \le n |\sigma_1-\sigma_0|^2$.
%		\item \it{Case 4}: $\sigma_1>\sigma_0 $ and $7\lVert\mu_1-\mu_0\rVert_2^2 \le n |\sigma_1-\sigma_0|^2$. 
%	\end{itemize}

	{\it Case 1: $\sigma_1\ge M_1\sigma_0$.}
	We define the test function $\varphi_{1,n}=\mathbbm 1\{\lVert y-\mu_0 \rVert_2^2/\sigma_0^2 > M_0^2 n\}$ for $M_0>0$.
	Under the law $y\sim P_{\mu_0,\sigma_0}$, $\lVert y-\mu_0 \rVert_2^2/\sigma_0^2 $ has a chi-squared distribution with $n$ degrees of freedom. Let $\chi_{k,\gamma}^2$ denote a chi-squared random variable with $k$ degrees of freedom and noncentrality parameter $\gamma$. It is known \citep[Lemma~1]{laurent2000adaptive} that chi-squared distributions satisfy the tail bounds: for any $t>0$, $\Pr\{\chi_{k,0}^2-k\ge 2\sqrt{kt} + 2t\}\le e^{-t}$ and $\Pr\{\chi_{k,0}^2-k\le -2\sqrt{kt}\}\le e^{-t}$.
	Therefore, if $M_0$ is sufficiently large, it follows that $\mathbb E_{\mu_0,\sigma_0}\varphi_{1,n}\le e^{-K_1n}\le e^{-K_1n\epsilon^2/\sigma_0^2}$ for some $K_1>0$, since $\epsilon\in(0,\sigma_0)$.
	Under the law $y\sim P_{\mu,\sigma}$, $\lVert y-\mu_0 \rVert_2^2/\sigma^2 $ has a noncentral chi-squared distribution with $n$ degrees of freedom and noncentrality parameter $\lVert \mu-\mu_0 \rVert_2^2/\sigma^2$. Thus, we obtain
$	\mathbb E_{\mu,\sigma}(1-\varphi_{1,n}) =\Pr\{\chi_{n,\lVert \mu-\mu_0 \rVert_2^2/\sigma^2}^2 \le  {\sigma_0^2 M_0^2 n}/{\sigma^2} \}
\le \Pr\{\chi_{n,0}^2\le {\sigma_0^2 M_0^2 n}/{\sigma^2}\}$,
	where we have used the fact that $\Pr\{\chi_{n,a}^2\le \cdot\}\le \Pr\{\chi_{n,b}^2\le \cdot\}$ for every $a\ge b$ (by the monotonicity of the Marcum Q-function). 
	We obtain $\sigma\ge \sigma_1 -|\sigma-\sigma_1|\ge M_1 \sigma_0 -\epsilon/6\ge M_1 \sigma_0 -\sigma_0/6\ge M_1\sigma_0/2$.
	Therefore, once $M_1$ is chosen to be sufficiently large to dominate $M_0$, the tail bound of the chi-squared distribution gives $\mathbb E_{\mu,\sigma}(1-\varphi_{1,n})\le e^{-K_2n}$ for some $K_2>0$, which is further bounded by $e^{-K_2n\epsilon^2/\sigma_0^2}$ since $\epsilon\in(0,\sigma_0)$.

	{\it Case 2: $\sigma_1 < M_1\sigma_0$ and $7\lVert\mu_1-\mu_0\rVert_2^2 > n |\sigma_1-\sigma_0|^2$.} We define the test
	$\varphi_{2,n}=\mathbbm 1\{(\mu_1-\mu_0)^T (y-\mu_0) > \lVert \mu_1-\mu_0\rVert_2^2/2 \}$.
	Since $8n^{-1}\lVert \mu_1-\mu_0\rVert_2^2\ge n^{-1}\lVert \mu_1-\mu_0\rVert_2^2 + |\sigma_1-\sigma_0|^2\ge \epsilon^2$, we easily obtain $\mathbb E_{\mu_0,\sigma_0} \varphi_{2,n}=\Phi(-\lVert \mu_1-\mu_0\rVert_2/(2\sigma_0))\le e^{- K_3 n \epsilon^2/\sigma_0^2}$ for some $K_3>0$, using the Gaussian concentration inequality. 
	For the type-II error, observe that
	$2(\mu_1-\mu_0)^T(\mu-\mu_0)=\lVert \mu_1-\mu_0\rVert_2^2+\lVert \mu-\mu_0\rVert_2^2-\lVert \mu-\mu_1\rVert_2^2$ and
	$ \lVert \mu-\mu_0\rVert_2^2 \ge \lVert\mu_1-\mu_0\rVert_2^2/2-\lVert \mu-\mu_1\rVert_2^2$. 
	It follows that
	\begin{align*}
		\mathbb E_{\mu,\sigma}(1-\varphi_{2,n})%&=P_{\mu,\sigma^2}\left\{ (\mu_1-\mu_0)^T (y-\mu)\le \lVert\mu_1-\mu_0\rVert_2^2/2 - (\mu_1-\mu_0)^T(\mu-\mu_0) \right\}\\
		&=P_{\mu,\sigma}\!\left\{ (\mu_1-\mu_0)^T (y-\mu)\le \lVert\mu-\mu_1\rVert_2^2/2 - \lVert\mu-\mu_0\rVert_2^2/2 \right\}\\
		%&\le P_{\mu,\sigma}\left\{ (\mu_1-\mu_0)^T (y-\mu)\le \lVert\mu-\mu_1\rVert_2^2- \lVert\mu_1-\mu_0\rVert_2^2/4 \right\}.
		&\le P_{\mu,\sigma}\!\left\{ \frac{(\mu_1-\mu_0)^T (y-\mu)}{\sigma\lVert\mu_1-\mu_0\rVert_2}\le \frac{\lVert\mu-\mu_1\rVert_2^2}{\sigma\lVert\mu_1-\mu_0\rVert_2}- \frac{\lVert\mu_1-\mu_0\rVert_2}{4\sigma} \right\}.
	\end{align*}
	As $\lVert\mu_1-\mu_0\rVert_2^2\ge n\epsilon^2/8$ and $\lVert\mu-\mu_1\rVert_2^2\le n \epsilon^2/36$, the rightmost side is bounded by $\Phi(-\sqrt{n}\epsilon /(72\sqrt{2}\sigma))$. Using the inequality $\sigma\le |\sigma-\sigma_1|+\sigma_1\le \epsilon/6 + M_1\sigma_0\le 2M_1\sigma_0$, the probability is further bounded by $\Phi(-\sqrt{n}\epsilon /(144\sqrt{2} M_1\sigma_0))\le e^{- K_4 n \epsilon^2/\sigma_0^2}$ for some $K_4>0$.
	
	{\it Case 3: $\sigma_0\le\sigma_1 < M_1 \sigma_0$ and $7\lVert\mu_1-\mu_0\rVert_2^2 \le n |\sigma_1-\sigma_0|^2$.} We define the test $\varphi_{3,n}=\mathbbm 1\{n^{-1}\lVert y-\mu_0 \rVert_1-\sqrt{2/\pi}\sigma_0> \sqrt{2/\pi}(\sigma_1-\sigma_0)/12\}$. Let $z \sim \text{N}_n(0_n,I_n)$ and observe that $\mathbb E\lVert z \rVert_1=n\sqrt{2/\pi}$. Then,
	\begin{align*}
		\mathbb E_{\mu_0,\sigma_0} \varphi_{3,n} = \Pr\!\left\{\lVert z \rVert_1-\mathbb E \lVert z \rVert_1> \frac{\sqrt{2/\pi}n}{12\sigma_0}(\sigma_1-\sigma_0)\right\}\le \Pr\!\left\{\lVert z \rVert_1-\mathbb E \lVert z \rVert_1> \frac{\sqrt{7}n\epsilon}{24\sqrt{\pi}\sigma_0}\right\},
%		\label{eqn:testeq1}
	\end{align*}
	since $\sigma_1\ge \sigma_0$ and $(1+1/7)|\sigma_1-\sigma_0|^2 \ge n^{-1}\lVert\mu_1-\mu_0\rVert_2^2 + |\sigma_1-\sigma_0|^2\ge \epsilon^2$. Note that $\lVert z \rVert_1-\mathbb E \lVert z \rVert_1$ is the sum of independent sub-Gaussian random variables with mean zero. It thus satisfies the concentration inequality of the form $\Pr\{|\lVert z \rVert_1-\mathbb E \lVert z\rVert_1 |>t\}\le e^{-K't^2/n}$ for any $t>0$ and some $K'>0$ \citep[Proposition 2.5]{wainwright2019high}. This gives $\mathbb E_{\mu_0,\sigma_0} \varphi_{3,n}\le e^{-K_5n\epsilon^2/\sigma_0^2}$ for some $K_5>0$.
	Observe also that 
	% by triangle inequality and Cauchy-Schwarz inequality,
	\begin{align*}
		\mathbb E_{\mu,\sigma}(1-\varphi_{3,n})&
	\le P_{\mu,\sigma}\!\left\{ n^{-1}\lVert y -\mu \rVert_1-\sqrt{2/\pi}\sigma_0 \le \frac{\sqrt{2/\pi}}{12}(\sigma_1-\sigma_0) + n^{-1/2} \lVert \mu-\mu_0 \rVert_2 \right\}\\
		& = \Pr\!\left\{ \lVert z \rVert_1-\mathbb E \lVert z \rVert_1 \le \frac{n}{\sigma}\left[\frac{\sqrt{2/\pi}}{12}(\sigma_1-\sigma_0) + \sqrt{2/\pi}(\sigma_0-\sigma) +n^{-1/2}\lVert \mu-\mu_0 \rVert_2 \right] \right\}.
	\end{align*}
	The term in the bracket is bounded by
	\begin{align*}
%		\begin{split}
			&-\frac{11}{12} \sqrt{2/\pi} (\sigma_1-\sigma_0) + \sqrt{2/\pi} (\sigma_1-\sigma) + n^{-1/2}\lVert \mu_1-\mu_0 \rVert_2 + n^{-1/2}\lVert \mu-\mu_1 \rVert_2\\
			&\quad\le\left(-\frac{11}{12}\sqrt{2/\pi} +\frac{1}{\sqrt{7}}\right) (\sigma_1-\sigma_0) + n^{-1/2}\lVert \mu-\mu_1 \rVert_2+|\sigma-\sigma_1| \\
			&\quad\le -\frac{7}{20}(\sigma_1-\sigma_0)+\sqrt{2(n^{-1}\lVert \mu-\mu_1\rVert_2^2 + |\sigma-\sigma_1|^2)},
%		\end{split}
%		\label{eqn:testeq2}
	\end{align*}
where we use the inequality $\sigma_1\ge \sigma_0$.
	Using the fact that $(1+1/7)|\sigma_1-\sigma_0|^2 \ge \epsilon^2$, the rightmost side can be further bounded by  $[-(7/20)\sqrt{7/8}+\sqrt{2}/6]\epsilon < -\epsilon/11$. Therefore, using the inequality $\sigma\le 2M_1\sigma_0$ (as in Case~2), we obtain $\mathbb E_{\mu,\sigma}(1-\varphi_{3,n})\le \Pr\{ \lVert z \rVert_1-\mathbb E \lVert z \rVert_1 \le - n\epsilon/(22 M_1\sigma_0)\}\le e^{-K_6n\epsilon^2/\sigma_0^2}$ for some $K_6>0$.

	{\it Case 4: $\sigma_0 > \sigma_1 $ and $7\lVert\mu_1-\mu_0\rVert_2^2 \le n |\sigma_1-\sigma_0|^2$.}
	Define the test function
	$\varphi_{4,n}=\mathbbm 1\{n^{-1}\lVert y-\mu_0 \rVert_1-\sqrt{2/\pi}\sigma_0\le \sqrt{2/\pi}(\sigma_1-\sigma_0)/12 \}$. 
	Since $\sigma_0>\sigma_1$ and $(1+1/7)|\sigma_1-\sigma_0|^2 \ge \epsilon^2$,
	we obtain $\mathbb E_{\mu_0,\sigma_0} \varphi_{4,n} 
	\le \Pr\{\lVert z \rVert_1-\mathbb E \lVert z \rVert_1 	\le {-\sqrt{7}n\epsilon}/(24\sqrt{\pi}\sigma_0)\}\le e^{-K_7n\epsilon^2/\sigma_0^2}$ for some $K_7>0$, similar to Case~3. Also observe that
	\begin{align*}
		\mathbb E_{\mu,\sigma}(1-\varphi_{4,n}) % &\le P_{\theta,\sigma^2}\left\{ n^{-1}\lVert Y-\theta \rVert_1-\sqrt{2/\pi} \sigma_0 > \frac{-\sqrt{2/\pi}}{12}(\sigma_0-\sigma_1) - n^{-1/2} \lVert \theta-\theta_0 \rVert_2 \right\}\\
		&\le \Pr\!\left\{ \lVert z \rVert_1-\mathbb E \lVert z \rVert_1 > \frac{n}{\sigma}\left[\frac{-\sqrt{2/\pi}}{12}(\sigma_0-\sigma_1)  + \sqrt{2/\pi}(\sigma_0-\sigma)  - n^{-1/2}\lVert \mu-\mu_0 \rVert_2 \right] \right\}.
	\end{align*}
	Similar to Case~3, using $\sigma_0>\sigma_1$, the term in the bracket is bounded below by
	\begin{align*}
		&\frac{11}{12} \sqrt{2/\pi} (\sigma_0-\sigma_1) + \sqrt{2/\pi} (\sigma_1-\sigma) - n^{-1/2}\lVert \mu_1-\mu_0 \rVert_2 - n^{-1/2}\lVert \mu-\mu_1 \rVert_2\\
		&\quad \ge  \frac{7}{20}(\sigma_0-\sigma_1)-\sqrt{2(n^{-1}\lVert \mu-\mu_1 \rVert_2^2 + |\sigma-\sigma_1|^2)}\\
		&\quad \ge \epsilon/11.
	\end{align*}
	Hence, using the inequality $\sigma \le |\sigma-\sigma_1|+\sigma_1 <  \epsilon/6 +\sigma_0 \le 7\sigma_0/6$, we have $\mathbb E_{\mu,\sigma}(1-\varphi_{4,n})\le \Pr\{ \lVert z \rVert_1-\mathbb E \lVert z \rVert_1 > 6n\epsilon/(77\sigma_0) \}\le e^{-K_8n\epsilon^2/\sigma_0^2}$ for some $K_8>0$.
	
	We can construct the test function $\varphi_n$ by choosing from $\varphi_{1,n}$ through $\varphi_{4,n}$, depending on the values of $(\mu_0,\sigma_0)$ and $(\mu_1,\sigma_1)$.
\end{proof}

	We compare our testing errors with those of \citet{ghosal2017fundamentals} and \citet{naulet2018some}. The testing error in Lemma~8.27 of \citet{ghosal2017fundamentals} involves $\sigma_1^2$, requiring the prior for $\sigma^2$ to be appropriately truncated. As a result, commonly used priors for $\sigma^2$ supported on $(0,\infty)$, such as inverse gamma priors, cannot be used directly. The testing error in Proposition~S2 of \citet{naulet2018some} is free of $\sigma_1^2$ but does not decay at a rate proportional to $\sigma_0^{-2}$. Therefore, their test function is not suitable for settings in which $\sigma_0^2$ varies with $n$. (Note that they measure the discrepancy between standard deviations on a logarithmic scale.) In contrast, our decay rate, given by $n\epsilon^2/\sigma_0^2$, properly accounts for such cases, as shown in Section~\ref{sec:example}.

It should be noted that Theorem~\ref{lmm:test} provides a test for $\epsilon$-separated balls only when $\epsilon\in(0,\sigma_0)$. This contrasts with the Hellinger metric and the $L_2$-metric in Gaussian models with known variance, which allow test functions for $\epsilon$-separated balls for every $\epsilon>0$. Consequently, Theorem~\ref{lmm:test} cannot be used for the general `shell approach' \citep[Theorem 8.12]{ghosal2017fundamentals} to obtain the refined version of contraction theory (see, e.g., Theorem~2.4 of \citet{ghosal2000convergence} or Theorem~8.11 of \citet{ghosal2017fundamentals}). Nevertheless, when combined with the metric entropy requirement, Theorem~\ref{lmm:test} is useful for constructing a global test over a sieve for the basic form of contraction theory (see, e.g., Theorem~2.1 of \citet{ghosal2000convergence} or Theorem~8.9 of \citet{ghosal2017fundamentals}). Below, we denote by $N(\epsilon,\mathcal E, \rho)$ the $\epsilon$-covering number of a semi-metric space $\mathcal E$ with respect to a semi-metric $\rho$.

\begin{theorem}[Global test]
	\label{lmm:globaltest}
	For a positive sequence $\epsilon_n>0$ such that $\epsilon_n / \sigma_0\rightarrow 0$ and $n\epsilon_n^2/\sigma_0^2\rightarrow\infty$, and a sieve $\Theta_n\subset \mathbb R^n\times (0,\infty)$,
	suppose that $N(\epsilon_n,\Theta_n,d)\le e^{D n\epsilon_n^2/\sigma_0^2}$ for some constant $D>0$.
	Then, for every $M>\max\{\sqrt{D/K},6\}$ with $K$ as in Theorem~\ref{lmm:test}, there exists a test $\varphi_n^\ast\in\{0,1\}$ such that
	\begin{align*}
		\mathbb E_{\mu_0,\sigma_0} \varphi_n^\ast \rightarrow 0,\quad 	\sup_{(\mu,\sigma)\in\Theta_n:d((\mu,\sigma),(\mu_0,\sigma_0))\ge M\epsilon_n}\mathbb E_{\mu,\sigma} (1-\varphi_n^\ast) \le e^{-KM^2n\epsilon_n^2/\sigma_0^2}.
%		\label{eqn:globaltest}
	\end{align*}
\end{theorem}

\begin{proof}
	Let $\{\mathcal C_1,\dots,\mathcal C_{N_n}\}$ be an $(M\epsilon_n/6)$-cover of $\{(\mu,\sigma)\in\Theta_n:d((\mu,\sigma),(\mu_0,\sigma_0))\ge M\epsilon_n\}$ with respect to the metric $d$, where $N_n$ denotes the covering number. Then, for each $\mathcal C_j$, with $\epsilon=M\epsilon_n$, the test from Theorem~\ref{lmm:test} can be constructed, provided that $n$ is sufficiently large so that $M\epsilon_n\le \sigma_0$. We denote each of these tests by $\varphi_{n,j}$, and define the combined test as $\varphi_n^\ast = \max_{1\le j\le N_n} \varphi_{n,j}$. Because
	\begin{align*}
		N_n=N\!\left(\frac{M\epsilon_n}{6},\left\{(\mu,\sigma)\in\Theta_n:d((\mu,\sigma),(\mu_0,\sigma_0))\ge M\epsilon_n\right\},d\right)\le N(\epsilon_n,\Theta_n,d),
	\end{align*}
	whenever $M > 6$, it follows that
%	\begin{align*}
$
		\mathbb E_{\mu_0,\sigma_0} \varphi_n^\ast \le \sum_{j=1}^{N_n}\mathbb E_{\mu_0,\sigma_0} \varphi_{n,j} \le N(\epsilon_n,\Theta_n,d)e^{-KM^2n\epsilon_n^2/\sigma_0^2}\le e^{-(KM^2-D)n\epsilon_n^2/\sigma_0^2}
$.
%	\end{align*} 
	The rightmost side goes to zero whenever $KM^2>D$.
	On the other hand,
	\begin{align*}
		\sup_{(\mu,\sigma)\in\Theta_n:d((\mu,\sigma),(\mu_0,\sigma_0))\ge M\epsilon_n}\mathbb E_{\mu,\sigma} (1-\varphi_n^\ast) 
		& \le \max_{1\le j\le N_n} \sup_{(\mu,\sigma)\in \mathcal C_j}\mathbb E_{\mu,\sigma} (1-\varphi_n^\ast)
		\\
		&\le \max_{1\le j \le N_n} \sup_{(\mu,\sigma)\in \mathcal C_j}\mathbb E_{\mu,\sigma} (1-\varphi_{n,j})
		\\
		&\le e^{-KM^2n\epsilon_n^2/\sigma_0^2},
	\end{align*}
	which concludes the proof.
\end{proof}	

The global test in Theorem \ref{lmm:globaltest} allows us to apply the testing-based approach for $L_2$-norm posterior contraction in the Gaussian model given by \eqref{eqn:model}. The following theorem builds on the results in \citet{ghosal2000convergence} and \citet{ghosal2007convergence}. % Since $\sigma^2$ is a univariate parameter independent of $n$, the conditions translate solely to those for the mean $\mu$. The inverse gamma prior below can be replaced by any prior that has a strictly positive density on $(0,\infty)$ and a polynomial tail on the right side.

\begin{theorem}[Posterior contraction]
	\label{thm:cont}
%	Suppose that $\mu$ and $\sigma^2$ are a priori independent and that the prior for $\sigma^2$ has a polynomial tail such that $\Pi\{\sigma^2 > t\}=O(t^{-1})$ for any large $t$.
	For $\epsilon_n>0$ such that $\epsilon_n/\sigma_0\rightarrow 0$ and  $n\epsilon_n^2/\sigma_0^2\rightarrow \infty$, suppose that there exists a sieve $\Theta_n\subset \mathbb R^n\times (0,\infty)$ such that for some constants $C>0$, $D>0$, and a sufficiently large constant $E>0$,	
	\begin{align}
		\begin{split}
		\Pi\{\lVert\mu-\mu_0\rVert_2^2\le n\epsilon_n^2, \, |\sigma^2-\sigma_0^2|\le \sigma_0\epsilon_n \}&\ge e^{-Cn\epsilon_n^2/\sigma_0^2},\\
N(\epsilon_n,\Theta_n,d)&\le e^{Dn\epsilon_n^2/\sigma_0^2},\\
\Pi\{(\mu,\sigma) \notin \Theta_n\} &\le e^{-En\epsilon_n^2/\sigma_0^2}.
		\end{split}
\label{eqn:cond0}
	\end{align}
	Then, there exists a constant $M>0$ such that the posterior satisfies
	$\mathbb E_{\mu_0,\sigma_0}\Pi\{(\mu,\sigma):d((\mu,\sigma),(\mu_0,\sigma_0))\ge M\epsilon_n\mid y\}\rightarrow 0$.
\end{theorem}

\begin{proof}
%Let $\tilde\epsilon_n = M_0\epsilon_n$ for a sufficiently large $ M_0>0$. To show that the posterior contraction rate is $\tilde\epsilon_n$ (or equivalently $\epsilon_n$), we
Our proof follows \citet{ghosal2000convergence} and \citet{ghosal2007convergence}. For $\delta>0$, we define the set
	\begin{align}
		\begin{split}
	\mathcal A_n(\delta)=\Bigg\{y\in\mathbb R^n: &\int \frac{p_{\mu,\sigma}}{p_{\mu_0,\sigma_0}}(y)d\Pi(\mu,\sigma) \\
	&\ge 	\Pi\!\left\{K(P_{\mu_0,\sigma_0},P_{\mu,\sigma})\le n\delta^2,V(P_{\mu_0,\sigma_0},P_{\mu,\sigma})\le n\delta^2\right\}e^{-2n\delta^2}\Bigg\},
	\label{eqn:kl}
		\end{split}
	\end{align}
where $K(P,Q) = \int \log({dP}/{dQ})dP$ and $V(P,Q) = \int(\log({dP}/{dQ})-K(P,Q))^2dP$.
For any $\delta\ge n^{-1/2}$, it is well known that $P_{\mu_0,\sigma_0}(\mathcal A_n(\delta))\ge 1-(4n\delta^2)^{-1}$; see, for example, Lemma~8.21 of \citet{ghosal2017fundamentals}.
Using the direct calculations of the Kullback-Leibler divergence $K$ and second-order variation $V$ (see, e.g., Theorem 9 of \citet{jeong2021unified}), and applying the Taylor expansion $\log(1+w)=w+O(w^2)$ and $1-(1+w)^{-1}=w+O(w^2)$ for $w=\sigma^2/\sigma_0^2-1$, which are valid when $\sigma^2/\sigma_0^2\rightarrow 1$,
we obtain
\begin{align*}
	K(P_{\mu_0,\sigma_0},P_{\mu,\sigma}) &= \frac{n}{2}\log\left(\frac{\sigma^2}{\sigma_0^2}\right)
	-\frac{n}{2}\left(1 -\frac{\sigma_0^2}{\sigma^2}\right) 
	+\frac{\lVert \mu - \mu_0\rVert _2^2}{2\sigma^2} = nO(w^2)+\frac{\lVert \mu - \mu_0\rVert _2^2}{2\sigma^2},\\
	V(P_{\mu_0,\sigma_0},P_{\mu,\sigma})	& = \frac{n}{2}\left(1 -\frac{\sigma_0^2}{\sigma^2}\right)^2 +\frac{\sigma_0^2\lVert \mu - \mu_0\rVert _2^2}{\sigma^4} = \frac{nw^2}{2}+ nO(w^3)+\frac{\sigma_0^2\lVert \mu - \mu_0\rVert _2^2}{\sigma^4}.
\end{align*}
	Therefore, for some $\tilde C>0$, we obtain
	$\Pi\{K(P_{\mu_0,\sigma_0},P_{\mu,\sigma})\le \tilde C n\epsilon_n^2/\sigma_0^2,V(P_{\mu_0,\sigma_0},P_{\mu,\sigma})\le \tilde C n\epsilon_n^2/\sigma_0^2\}\ge\Pi\{\lVert\mu-\mu_0\rVert_2^2\le  n\epsilon_n^2, |\sigma^2/\sigma_0^2-1|\le \epsilon_n/\sigma_0\}$, which is further bounded below by $e^{-Cn\epsilon_n^2/\sigma_0^2}$.
	Let $\mathcal B_n=\{(\mu,\sigma):d((\mu,\sigma),(\mu_0,\sigma_0))\ge M\epsilon_n\}$ for a large $M>0$.
	It is easy to see that	
	\begin{align*}
\mathbb E_{\mu_0,\sigma_0} \Pi( \mathcal B_n\mid y) &\le\mathbb E_{\mu_0,\sigma_0} \varphi_n^\ast +   \mathbb E_{\mu_0,\sigma_0} \mathbbm 1(\mathcal A_n(\tilde C^{1/2}\epsilon_n/\sigma_0)^c) \\
&\quad + e^{(2\tilde C+C)n\epsilon_n^2/\sigma_0^2}\mathbb E_{\mu_0,\sigma_0}\int_{\mathcal B_n} (1-\varphi_n^\ast)\frac{p_{\mu,\sigma}}{p_{\mu_0,\sigma_0}}(y)d\Pi(\mu,\sigma),
	\end{align*}
where $\varphi_n^\ast$ denotes the test in Lemma~\ref{lmm:globaltest}.
The first term on the right-hand side tends to zero according to Lemma~\ref{lmm:globaltest}. The second term also goes to zero based on the probability bound of \eqref{eqn:kl}. For the third term, observe by Theorem~\ref{lmm:globaltest} that $\mathbb E_{\mu_0,\sigma_0}\int_{\mathcal B_n} (1-\varphi_n^\ast)({p_{\mu,\sigma}}/{p_{\mu_0,\sigma_0}})(y)d\Pi(\mu,\sigma)\le e^{-KM^2n\epsilon_n^2/\sigma_0^2} + \Pi\{(\mu,\sigma) \notin \Theta_n\}$, which dominates the term $e^{(2\tilde C+C)n\epsilon_n^2/\sigma_0^2}$ if $M$ and $E$ are sufficiently large. This concludes the proof.
\end{proof}

Although Theorem~\ref{thm:cont} provides a general framework for posterior contraction, its conditions are somewhat abstract. By assuming that $\sigma^2$ is a priori independent of $\mu$ and imposing specific conditions on its prior, we derive stronger yet more convenient sufficient conditions.

\begin{theorem}[Posterior contraction; sufficient conditions]
	\label{thm:suff}
		Suppose that $\mu$ and $\sigma^2$ are a priori independent, and that $\sigma_0^2>n^{-B}$ for some constant $B>0$.
	Assume that the prior for $\sigma^2$ has a polynomial tail such that $\Pi\{\sigma^2 > t\}=O(t^{-1})$ for any large $t$, and that its density $g$ is $L$-Lipschitz on $(0,\infty)$.
	For $\epsilon_n>0$ such that $\epsilon_n/\sigma_0\rightarrow 0$ and  $n\epsilon_n^2/\sigma_0^2\ge \log n$, suppose that $g$ satisfies 
		\begin{align}
		g(\sigma_0^2)\ge 2L\sigma_0\epsilon_n,
		\label{eqn:gbo}
	\end{align} 
and that there exists a sieve $\mathcal M_n\subset \mathbb R^n$ such that for some constants $\bar C >0$, $\bar D >0$, and a sufficiently large constant $\bar E >0$,	
	\begin{align}
		\begin{split}
			\Pi\{\lVert\mu-\mu_0\rVert_2^2\le n\epsilon_n^2\}&\ge e^{- \bar C n\epsilon_n^2/\sigma_0^2},\\
			N(\sqrt{n}\epsilon_n,\mathcal M_n,\lVert\cdot\rVert_2)&\le e^{ \bar D n\epsilon_n^2/\sigma_0^2},\\
			\Pi\{\mu \notin \mathcal M_n\} &\le e^{- \bar E n\epsilon_n^2/\sigma_0^2}.
		\end{split}
		\label{eqn:cond}
	\end{align}
	Then, there exists a constant $M>0$ such that the posterior satisfies
	$\mathbb E_{\mu_0,\sigma_0}\Pi\{(\mu,\sigma):d((\mu,\sigma),(\mu_0,\sigma_0))\ge M\epsilon_n\mid y\}\rightarrow 0$.
\end{theorem}

\begin{proof}
	Using \eqref{eqn:gbo} and the Lipschitz continuity of $g$, we observe that
		\begin{align*}
		\Pi\{|\sigma^2-\sigma_0^2|\le \sigma_0\epsilon_n\} \ge  2\sigma_0\epsilon_n \inf_{s:|s-\sigma_0^2|\le\sigma_0\epsilon_n } g(s) \ge 2\sigma_0\epsilon_n[g(\sigma_0^2)-L\sigma_0\epsilon_n]\ge 2L\sigma_0^2\epsilon_n^2 \ge \frac{2L\log n}{n^{2B+1}}.
	\end{align*}
From the a priori independence and the first condition of \eqref{eqn:cond}, the first condition of \eqref{eqn:cond0} holds, since $n\epsilon_n^2/\sigma_0^2\ge \log n$.
Next, we choose $\Theta_n=\sqrt{2}\mathcal M_n\times (0,\sqrt{2}e^{R n\epsilon_n^2/\sigma_0^2}]$ for a large constant $R>0$. Then, $N(\epsilon_n,\Theta_n,d) \le N(\sqrt{n}\epsilon_n,\mathcal M_n,\lVert
\cdot\rVert_2) \times N(\epsilon_n,(0,e^{Rn\epsilon_n^2}],\lvert\cdot\rvert)$,
which is bounded as required by the second condition of \eqref{eqn:cond0}, since $n\epsilon_n^2/\sigma_0^2\ge \log n$.
Finally, observe that $\Pi\{(\mu,\sigma^2)\notin \Theta_n\} \le \Pi\{\mu\notin \sqrt{2} \mathcal M_n\} + \Pi\{\sigma^2 > 2e^{2Rn\epsilon_n^2/\sigma_0^2}\}$. Since the prior of $\sigma^2$ has a polynomially decaying right tail, the last condition of \eqref{eqn:cond0} is satisfied provided that $\bar E$ and $R$ are sufficiently large.
\end{proof}

In summary, the prior for $\sigma^2$ should be $L$-Lipschitz and exhibit a polynomially decaying right tail, and its density at the true $\sigma_0^2$ should satisfy the lower bound in \eqref{eqn:gbo}. The polynomial decay is certainly a mild requirement. If $\sigma_0^2$ is bounded away from both zero and infinity, the condition in \eqref{eqn:gbo} is automatically satisfied whenever $g$ is independent of $n$ and is supported on $(0,\infty)$. To accommodate decreasing or increasing $\sigma_0^2$, we consider the following two widely used priors for variance parameters. Below, $a_n\ll b_n$ means $a_n/b_n\rightarrow 0$.

	\begin{example}[Inverse gamma]
	Suppose that $\epsilon_n/\sigma_0 \le n^{-\xi}$ for some $\xi\in(0,1/2)$, and let $g(s)\propto s^{-a-1}\exp(-b/s)$ denote the inverse gamma density with shape $a>0$ and scale $b>0$. It is easy to see that $g$ is Lipschitz continuous, exhibits a polynomially decaying right tail, and that \eqref{eqn:gbo} holds if $b/(\xi\log n)\le \sigma_0^2\ll n^{\xi/(a+2)}$.
\end{example}
	
	\begin{example}[Half-Cauchy]
Suppose that $\epsilon_n/\sigma_0 \le n^{-\xi}$ for some $\xi\in(0,1/2)$, and let $g(s)\propto (r^2+s^2)^{-1}$ denote the half-Cauchy density with scale $r>0$. It is straightforward to verify that $g$ is Lipschitz continuous, has a polynomially decaying right tail, and that \eqref{eqn:gbo} is satisfied if $\sigma_0^2\ll n^{\xi/2}$.
\end{example}

\section{Application}
\label{sec:example}

\subsection{High-dimensional regression}
\label{sec:highreg}

Consider a high-dimensional regression model with $\mu_0=X\beta_0$, where $X\in\mathbb R^{n\times p}$ is the design matrix and $\beta_0\in\mathbb R^p$ is a sparse vector. In other words, we assume $p>n$ but many components of $\beta_0$ are zero. 
We denote by $\lVert\cdot\rVert_{\infty}$ the maximum norm for both vectors and matrices.
The support of the true signal $S_0$ and its cardinality $|S_0|$ are unknown. When $\sigma_0^2$ is known, the contraction rate of this model was studied by \citet{castillo2015bayesian} under mild conditions. Here, we extend the analysis to the case of an unknown $\sigma_0^2$ under the additional restrictions that $\lVert\beta_0\rVert_\infty=O(\log p)$ and $\lVert X\rVert_\infty=O(\log p)$.
We adopt a prior for $\sigma^2$ that satisfies the conditions for Theorem~\ref{thm:suff}, and assume
that $n^{-B}\le \sigma_0^2\le n^{B}$ for some $B>0$.
We put a prior on the support $S$ such that $\Pi\{S=\tilde S\} \propto \binom{p}{|\tilde S|}^{-1} e^{-A|\tilde S|\log p}$ for a constant $A>0$.
Denoting by $\beta_S$ the components of $\beta$ in $S$, we use the priors $\beta_S\mid S \sim N_{|S|}(0_{|S|},\tau^2 I_{|S|})$ for $\tau^2>0$ and $\beta_{S^c}|S\sim\delta_0$, where $\delta_0$ is the point mass at zero. The target rate is $\epsilon_n=\sigma_0\sqrt{(|S_0|\log p)/n}$. We assume
$|S_0|>0$ and $(|S_0|\log p)/n\rightarrow 0$ such that
$\epsilon_n/\sigma_0\rightarrow 0$ and $n\epsilon_n^2/\sigma_0^2\ge \log n$.

First, observe that
\begin{align*}
\Pi\{\lVert X(\beta-\beta_0) \rVert_2^2 \le n\epsilon_n^2\}&\ge \Pi\{S=S_0\}\Pi\{\lVert X_S(\beta_S-\beta_{0,S}) \rVert_2^2 \le n\epsilon_n^2\mid S=S_0\} \\
&\ge \Pi\{S=S_0\}\Pi\{\lVert \beta_S-\beta_{0,S} \rVert_2^2 \le \epsilon_n^2/(\lVert X\rVert_\infty^2|S|)\mid S=S_0\},
\end{align*}
where the last inequality follows from $\lVert X_S\beta_S\rVert_2\le \sqrt{n|S|}\times\lVert X\rVert_\infty \lVert\beta_S\rVert_2$. Following the calculation in Section~7.7 of \citet{ghosal2007convergence} and using the lower bound for the chi-squared distribution function, $\Pr\{\chi_{k,0}^2 \le t\}\ge (t/2)^{k/2}e^{-t/2}/\Gamma(k/2+1)$ for $t>0$, it can be shown that
\begin{align*}
\Pi\{\lVert \beta_S-\beta_{0,S} \rVert_2^2 \le \epsilon_n^2/(\lVert X\rVert_\infty^2|S|)\mid S=S_0\}&\ge 2^{-|S_0|/2}e^{-\lVert\beta_0\rVert_2^2}\Pi\{\lVert \beta_S \rVert_2^2 \le \epsilon_n^2/(2\lVert X\rVert_\infty^2|S|)\mid S=S_0\}\\
&\ge e^{-C_1 |S_0|\log p},
\end{align*}
for some $C_1>0$. Since $\Pi\{S=S_0\}\ge e^{-C_2 |S_0|\log p}$ for some $C_2>0$, the prior concentration condition follows.
Next, choose the sieve $\mathcal M_n= \{\mu=X_S\beta_S: \lVert \beta \rVert_\infty\le n,|S|\le M_0|S_0|\}$ for a sufficiently large $M_0>0$.
Noting that $\lVert X_S\beta_S\rVert_2\le \sqrt{n}|S|\times \lVert X\rVert_\infty\lVert\beta_S\rVert_\infty$,
we obtain
\begin{align*}
N(\sqrt{n}\epsilon_n, \mathcal M_n ,\lVert\cdot\rVert_2)&\le \sum_{S:|S|\le M_0|S_0|}N\!\left(\frac{\epsilon_n}{\lVert X\rVert_\infty|S|}, \{\lVert \beta_S \rVert_\infty\le n\} ,\lVert\cdot\rVert_\infty\right)\\
&\le \binom{p}{M_0|S_0|} \left(\frac{3\lVert X\rVert_\infty |S_0|n}{\epsilon_n}\right)^{M_0|S_0|}\\
&\le e^{C_3|S_0|\log p},
\end{align*}
for some $C_3>0$, thereby verifying the entropy condition. Lastly, observe that
\begin{align*}
	\Pi\{\mu \notin \mathcal M_n\}& \le \Pi\{|S|>M_0|S_0|\} + \sum_{\tilde S:|\tilde S|\le M_0|S_0|}\Pi\{S=\tilde S\}\Pi\{\lVert\beta_S\rVert_\infty >n \mid S=\tilde S\}\\
	& \le 2e^{-A M_0|S_0|\log p} + 2 M_0|S_0|e^{-n^2/(2\tau^2)},
\end{align*}
which is further bounded by $e^{-C_4 |S_0|\log p}$ for a sufficiently large $C_4$, provided that 
$M_0$ is chosen to be sufficiently large. Therefore, the posterior contraction rate is $\epsilon_n$ with respect to $n^{-1/2}\lVert X(\beta-\beta_0) \rVert_2 + |\sigma-\sigma_0|$. The rate for $\lVert \beta-\beta_0 \rVert_2$ can be obtained by imposing suitable compatibility conditions (see, e.g., \citet{castillo2015bayesian} and \citet{jeong2021posterior}).

\subsection{Adaptive nonparametric regression}
\label{sec:nonreg}

We consider nonparametric regression with the target function $f_0$
belonging to the $\alpha$-smooth H\"older space $\mathcal H^\alpha$ over a bounded domain, for $\alpha>0$ (see Definition~C.4 of \citet{ghosal2017fundamentals}). 
Although various Bayesian nonparametric methods have been proposed, including Gaussian processes \citep{van2008rates}, we adopt a basis expansion approach \citep{shen2015adaptive}. For simplicity, we focus on the univariate case with $p=1$ although our approach can be readily extended to multivariate settings. To approximate $f_0$, we employ B-spline basis functions, though alternative bases, such as the Fourier basis or wavelets \citep{donoho1995wavelet}, can also be used. Let $\psi_J$ be the $q$-degree B-spline basis of dimension $J$ with $K$ interior uniform knots, such that $J=q+K+1$, and let $\beta_J$ be the corresponding coefficients; that is, $f=\psi_J^T\beta_J$ and $\mu=B_J\beta_J$, where $B_J\in\mathbb R^{n\times J}$ is the basis matrix. 
For a precise definition of B-splines and their use in Bayesian nonparametrics, refer to Section~D of \citet{ghosal2017fundamentals} and \citet{shen2015adaptive}.
We assign priors $\Pi\{J=j\}\propto e^{-Aj\log j}$ for $A>0$, $\beta_J\mid J\sim \text{N}_J(0_J,\tau^2 I_J)$ for $\tau^2>0$, and a prior for $\sigma^2$ that satisfies the conditions for Theorem~\ref{thm:suff}. 
The target rate is $\epsilon_n = ((\sigma_0^2\log n)/n)^{\alpha/(2\alpha+1)}$, for which we assume that 	$((\log n)/n)^{2\alpha} \ll \sigma_0^2 \le n /\log n$ so that $\epsilon_n/\sigma_0\rightarrow 0$ and $n\epsilon_n^2/\sigma_0^2\ge \log n$.

Observe that $\lVert\mu-\mu_0\rVert_2 = \sqrt{n}\lVert f-f_0 \rVert_n$, where $\lVert f \rVert_n =\sqrt{ n^{-1}\sum_{i=1}^n|f(x_i)|^2}$ denotes the empirical $L_2$-norm. Classical approximation theory shows that for any $f_0\in\mathcal H^\alpha$, provided that $q+1\ge \alpha$, there exists $\hat\beta_J$ with bounded $\lVert\hat\beta_J\rVert_\infty$ such that $\lVert \psi_J^T\hat\beta_J-f_0\rVert_\infty\le C_1 \lVert f_0\rVert_{\mathcal H^\alpha} J^{-\alpha}$ for some $C_1>0$ \citep[e.g,][]{de1978practical}, where $\lVert\cdot\rVert_\infty$ denotes the supremum norm of a function (with a slight abuse of notation) and $\lVert \cdot\rVert_{\mathcal H^\alpha}$ denotes the H\"older norm. We assume that $\lVert f_0\rVert_{\mathcal H^\alpha}$ is bounded. Since the B-splines satisfy the sum-to-unity property $\psi_J^T 1_J=1$, it follows that $\lVert \psi_J^T\beta_J\rVert_\infty\le \lVert \beta_J\rVert_\infty$.
Therefore, for $\hat J=\lfloor C_2(n/(\sigma_0^2\log n))^{1/(2\alpha+1)}\rfloor$ with a sufficiently large $C_2>0$, 
\begin{align*}
\Pi\{\lVert f-f_0\rVert_n\le \epsilon_n\}&\ge \Pi\{J=\hat J\}\Pi\{\lVert f- \psi_J^T\hat\beta_J
\rVert_\infty\le \epsilon_n/2 \mid J=\hat J\}\\
&\ge \Pi\{J=\hat J\}\Pi\{\lVert \beta_J-\hat \beta_J\rVert_2\le \epsilon_n/(2\sqrt{J})\mid J=\hat J\}.
\end{align*}
Similar to Section~\ref{sec:highreg}, this is bounded below by $e^{-C_3\hat J \log n}$ for some $C_3>0$, since $\lVert\hat\beta_{\hat J}\rVert_\infty$ is bounded. Next, choose the sieve $\mathcal M_n=\{\mu=B_J \beta_J: \lVert \beta_J \rVert_\infty\le n, J\le M_0\hat J\}$ for a sufficiently large $M_0$. Using the inequality
$\lVert B_J\beta_J\rVert_2 = \sqrt{n}\lVert \psi_J^T\beta_J\rVert_n\le \sqrt{n}\lVert\beta_J\rVert_\infty$, the remaining conditions can be easily verified as in Section~\ref{sec:highreg}. Consequently, the contraction rate is $\epsilon_n$ with respect to $\lVert f-f_0 \rVert_n+|\sigma-\sigma_0|$.

%\begin{align*}
%N(\sqrt{n}\epsilon_n, \{\mu=B_J \beta_J: \lVert \beta_J \rVert_\infty\le n, J\le M_0 \hat J\} ,\lVert\cdot\rVert_2)\le\sum_{J=1}^{M_0\hat J}N(\epsilon_n, \{\beta_J\in\mathbb R^J : \lVert \beta_J \rVert_\infty\le n\} ,\lVert\cdot\rVert_\infty)\le M_0 \hat J\left(\frac{3n}{\epsilon_n}\right)^{M_0\hat J}
%\end{align*}

%\begin{align*}
%\Pi\{\mu \notin \mathcal M_n\}\le \Pi\{J > M_0 \hat J\}+\sum_{j=1}^{M_0 \hat J} \Pi\{J=j\} \Pi\{\lVert \beta_J \rVert_\infty >  n^{R}\mid J=j\} \le 2 e^{-M_0 \hat J\log \hat J}+ 2M_0 \hat Je^{-n^{2R}/(2\tau^2)}
%\end{align*}

\section{Discussion}
\label{sec:disc}
In this study, we establish a local test function for the general Gaussian model in \eqref{eqn:model} and derive sufficient conditions for posterior contraction using a global test based on metric entropy. Although our applications focus on high-dimensional regression and nonparametric regression, the framework also accommodates other models, including sparse mean models and change-point models. The techniques for constructing our local test may be extended to other Gaussian models, such as time series and white noise models (see Sections~D.6 and D.7 of \citet{ghosal2007convergence}).

\section*{Acknowledgment}
This research was supported by the National Research Foundation of Korea (NRF) grant funded by the Korean government (MSIT) (2022R1C1C1006735, RS-2023-00217705).

\bibliographystyle{apalike}
\bibliography{refGauReg.bib}

\end{document}